\documentclass[11pt,reqno]{amsart}

\usepackage{amsmath}
\usepackage{amssymb}
\usepackage{amscd}
\usepackage{latexsym}
\usepackage{graphicx}
\usepackage{color}

\newcommand{\ispa}[1]{\langle \,#1 \,\rangle } 
\newcommand{\tf}[2]{{\textstyle \frac{#1}{#2}}}

\newcommand{\ol}{\overline}

\newcommand{\mb}{\mathbb}

\newcommand{\dcal}{\mathcal{D}}

\newcommand{\fcal}{\mathcal{F}}

\newcommand{\hcal}{\mathcal{H}}

\newcommand{\lcal}{\mathcal{L}}

\newcommand{\tcal}{\mathcal{T}}

\newcommand{\spec}{{\rm Spec\,}}
\newcommand{\dsp}{\displaystyle}

\newcommand{\wlim}{\operatornamewithlimits{{\rm w}-lim}}

\renewcommand{\arccos}{{\rm Cos}^{-1}}
\renewcommand{\arcsin}{{\rm Sin}^{-1}}

\newtheorem{thm}{{\sc Theorem}}

\newtheorem{prop}[thm]{{\sc Proposition}}

\newtheorem{ex}{{\sc Example}}
\begin{document}

\title[The Hamiltonians of quantum walks]
{The Hamiltonians generating one-dimensional\\ 
discrete-time quantum walks}
\author{Tatsuya Tate}
\address{Mathematical Institute, Graduate School of Sciences, Tohoku University, 
Aoba, Sendai 980-8578, Japan. }
\email{tate@math.tohoku.ac.jp}
\thanks{The author is partially supported by JSPS Grant-in-Aid for Scientific Research (No. 21740117, No. 25400068).}
\date{\today}

\renewcommand{\thefootnote}{\fnsymbol{footnote}}
\renewcommand{\theequation}{\thesection.\arabic{equation}}
\renewcommand{\labelenumi}{{\rm (\arabic{enumi})}}
\renewcommand{\labelenumii}{{\rm (\alph{enumii})}}
\numberwithin{equation}{section}

\begin{abstract}
An explicit formula of the Hamiltonians generating one-dimensional discrete-time quantum walks is given. 
The formula is deduced by using the algebraic structure introduced in \cite{T}. 
The square of the Hamiltonian turns out to be an operator without, essentially, the `coin register', 
and hence it can be compared with the one-dimensional continuous-time quantum walk. 
It is shown that, under a limit with respect to a parameter, which expresses 
the magnitude of the diagonal components of the unitary matrix defining the discrete-time 
quantum walks, the one-dimensional continuous-time quantum walk is obtained from operators 
defined through the Hamiltonians of the one-dimensional discrete-time quantum walks. 
Thus, this result can be regarded, in one-dimension, as a partial answer to a problem 
proposed by Ambainis \cite{A}. 
\end{abstract}

\maketitle

\section{Introduction}\label{INTRO}
The notion of discrete-time quantum walks (quantum walks for short) are originally proposed 
in quantum physics by Aharonov-Davidovich-Zagury \cite{ADZ} and 
re-discovered in computer science by, for instance, Nayak-Vishwanath \cite{NV}, 
Aharonov-Ambainis-Kempe-Vazirani \cite{AAKV},  Ambainis-Bach-Nayak-Vishwanath-Watrous \cite{ABNVW}. 
For more historical background, see, for instance, \cite{Ke}, \cite{Ko3}. 
The one-dimensional quantum walks are defined as a non-commutative analogue of the usual 
random walks on the set of integers, $\mb{Z}$, and they are defined as unitary operators on the Hilbert space $\ell^{2}(\mb{Z},\mb{C}^{2})$ 
consisting of all square summable $\mb{C}^{2}$-valued functions on $\mb{Z}$. 
According to the asymptotic formulas obtained in \cite{ST}, there is a resemblance between 
the asymptotic behavior, in a long-time limit, of transition probabilities of one-dimensional quantum walks and 
the asymptotic behavior, in a high-energy limit, of modulus squares of the Hermite functions on the real line. 
Thus, it would be reasonable to think one-dimensional quantum walks as a discretized, 
in both of space and time parameters, model for one-dimensional quantum systems. 
However, concrete formulas of their Hamiltonian had not been made clear. 
The aim of this paper is to give an explicit formula for the Hamiltonians generating the one-dimensional 
quantum walks and investigate their properties. As a result, a direct relation between operators 
defined through the Hamiltonians of the discrete-time quantum walks and the continuous-time quantum walk
in one-dimension is obtained.

To describe our main results let us prepare notation. 
Let $T$ be a two-by-two special unitary matrix. The quantum walk associated with $T$, denoted $U(T)$, 
is defined, as a unitary operator on $\ell^{2}(\mb{Z},\mb{C}^{2})$, by the formula
\begin{equation}\label{1DQWdef}
U(T)=TP_{1} \tau +TP_{2}\tau^{-1}, 
\end{equation}
where $\tau$ is the shift operator on $\ell^{2}(\mb{Z},\mb{C}^{2})$ defined by 
\begin{equation}\label{shift}
(\tau f)(x)=f(x-1) \quad (f \in \ell^{2}(\mb{Z},\mb{C}^{2})), 
\end{equation}
and $P_{i}:\mb{C}^{2} \to \mb{C}^{2}$ denotes the orthogonal projection onto the 
one-dimensional subspace $\mb{C}{\bf e}_{i}$ in $\mb{C}^{2}$, where $\{{\bf e}_{1},{\bf e}_{2}\}$ denotes 
the standard basis on $\mb{C}^{2}$. 
Suppose that the special unitary matrix $T$ is given by 
\begin{equation}\label{su2}
T=
\begin{pmatrix}
a & b \\
-\ol{b} & \ol{a}
\end{pmatrix}
\quad \mbox{with} \quad ab \neq 0. 
\end{equation}
We set 
\begin{equation}\label{not1}
s=|a|, \ \  r=|b|, \ \ 
\alpha=\frac{a}{|a|},\ \  \beta=\frac{b}{|b|}. 
\end{equation}
We define the function $\varphi_{s}(\theta)$ in $\theta \in \mb{R}$ by 
\begin{equation}\label{eigenF}
\varphi_{s}(\theta)=\arccos (s \cos \theta). 
\end{equation}
For any integer $x$, we define the integrals $I(x)$, $J(x)$ by 
\begin{equation}\label{intco1}
I(x)=\frac{1}{2\pi} \int_{0}^{2\pi} e^{-ix\theta} \sin \theta \frac{\varphi_{s}(\theta)}{\sin \varphi_{s}(\theta)}\,d\theta,\quad 
J(x)=\frac{1}{2\pi} \int_{0}^{2\pi} e^{-ix \theta} \frac{\varphi_{s}(\theta)}{\sin \varphi_{s}(\theta)}\,d\theta. 
\end{equation}
We define the matrix-valued function $\dcal_{T}$ on $\mb{Z}$ by 
\begin{equation}\label{kernel1}
\dcal_{T}(x)=
\begin{pmatrix}
s\alpha^{x} I(x) & -it\alpha^{x+1} \beta J(x+1) \\
it \alpha^{x-1} \beta^{-1}J(x-1) & -s \alpha^{x}I(x)
\end{pmatrix}
\quad (x \in \mb{Z}). 
\end{equation}
We then define the bounded self-adjoint operator $D(T)$ on $\ell^{2}(\mb{Z},\mb{C}^{2})$ by the formula 
\begin{equation}\label{HamDef}
D(T)=\sum_{y \in \mb{Z}} \dcal_{T}(y)\tau^{y}.
\end{equation}
\begin{thm}\label{HamTH}
We have $\dsp U(T)=e^{iD(T)}$. 
\end{thm}
That is, the bounded self-adjoint operator $D(T)$ is the Hamiltonian generating the quantum walk $U(T)$. 
The proof of Theorem $\ref{HamTH}$ is rather easy, but the method to deduce the formula $\eqref{HamDef}$ might not be so obvious. 
Indeed, we use the algebraic structure behind one-dimensional quantum walks introduced 
in \cite{T}. This algebraic structure is nothing but a unitary representation of the infinite dihedral group (\cite{O}). 
See Section $\ref{ALG}$ in this paper for this viewpoint. 

Having obtained the Hamiltonian generating the one-dimensional quantum walk, it would be important to understand 
its various properties. First, we mention some of simple properties of the operator $D(T)$. 
The following can be proved in an easy way. (See Section $\ref{HAM}$ for details.) 
\begin{itemize}
\item The spectrum, $\spec (D(T))$, is given by 
\begin{equation}\label{SpecDT}
\spec(D(T))=[-\arccos (-s),-\arccos(s)] \cup [\arccos(s),\arccos(-s)].
\end{equation}
\item $D(T)$ does not have the locality, that is, $D(T)(\delta_{x} \otimes \varphi)$ has an infinite support, where 
the function $\delta_{x} \otimes u$ with $x \in \mb{Z}$ and $u \in \mb{C}^{2}$ is defined by 
\[
(\delta_{x} \otimes u) (y)=
\begin{cases}
u & \mbox{(when $y=x$)}, \\
0 & \mbox{(when $y \neq x$)}. 
\end{cases}
\]
\item The square $D(T)^{2}$ is, essentially, an operator acting on scalars (without `the coin register'), that is, we have the following. 
\begin{equation}\label{DTsq}
D(T)^{2}=\sum_{y \in \mb{Z}} \alpha^{y} \fcal(\varphi_{s}^{2})(y) I_{2} \tau^{y}, 
\end{equation}
where $I_{2}$ is the two-by-two identity matrix and 
$\fcal: L^{2}(S^{1}) \to \ell^{2}(\mb{Z})$ is the Fourier transform defined by 
\begin{equation}\label{SFourier}
\fcal(k)(x)=\frac{1}{2\pi}\int_{0}^{2\pi} e^{-ix \theta} k(\theta)\,d\theta \quad (x \in \mb{Z}), 
\end{equation}
where $L^{2}(S^{1})$ denotes the Hilbert space consisting of all square integrable (with respect to the uniform measure) functions 
on the unit circle $S^{1}$. 
\end{itemize}
According to the last item of the above, it would be reasonable to consider the operator $\widehat{H}(\varphi_{s})$ on 
the space $\ell^{2}(\mb{Z})$ of square summable (scalar) functions on $\mb{Z}$ defined by 
\begin{equation}\label{phis}
\widehat{H}(\varphi_{s}):=\sum_{y} \fcal(\varphi_{s})(y)\tau^{y}, 
\end{equation}
where $\tau$ is now the shift operator on $\ell^{2}(\mb{Z})$ defined by the same formula as in $\eqref{shift}$. 
In \cite{A}, Ambainis proposed a problem how discrete- and continuous-time 
quantum walks can be obtained one from another. 
Since we found the operators $\widehat{H}(\varphi_{s})$ acting on $\ell^{2}(\mb{Z})$ from 
the Hamiltonians of the one-dimensional discrete-time quantum walks, 
it would be interesting to consider some relationships between these operators and the continuous-time quantum walk. 
Recall that the continuous-time quantum walk, introduced originally by Childs-Farhi-Gutmann \cite{CFG}, 
is the unitary operator $e^{it \widehat{H}}$ acting on $\ell^{2}(\mb{Z})$, 
where $\widehat{H}$ denotes the standard symmetric random walk on $\mb{Z}$ defined by 
\begin{equation}\label{SRW}
\widehat{H}=\frac{1}{2}(\tau +\tau^{-1}).  
\end{equation}
Note that we have $\spec(\widehat{H})=[-1,1]$ while $\spec (\widehat{H}(\varphi_{s}))=[\arccos(s),\arccos(-s)]$. 
Therefore, it would be reasonable to scale the function $\varphi_{s}$ as 
\begin{equation}\label{scaleF}
\psi_{s}(\theta):=\frac{1}{\arcsin(s)}\arcsin(s \cos \theta)=\frac{1}{\arcsin(s)} \left(
\frac{\pi}{2} -\varphi_{s}(\theta)
\right)
\end{equation}
so that $\spec(\widehat{H}(\psi_{s}))=[-1,1]$, 
where $\widehat{H}(\psi_{s})$ is the operator defined by the formula $\eqref{phis}$ with $\psi_{s}$ replacing $\varphi_{s}$. 
We then compare the operator $\widehat{H}$ with $\widehat{H}(\varphi_{s})$ and $\widehat{H}(\psi_{s})$. 
Note that if we set 
\begin{equation}\label{psi0}
\psi_{0}(\theta)=\cos \theta \quad (\theta \in \mb{R}), 
\end{equation}
then one has $\widehat{H}(\psi_{0})=\widehat{H}$, and by Taylor's formula for $\arcsin (x)$, we see 
\[
\|\psi_{s}-\psi_{0}\|_{C(S^{1})}=O(s^{2}), \quad 
\|\widehat{H}(\psi_{s}) -\widehat{H}(\psi_{0})\|_{{\rm op}} =O(s^{2}) \quad (s \to 0), 
\]
where $\dsp \| \cdot \|_{C(S^{1})}$ denotes the supremum norm on the algebra $C(S^{1})$ 
consisting of all continuous functions on $S^{1}$ and $\|\cdot \|_{{\rm op}}$ denotes 
the operator norm.  Thus, in the above sense, the classical random walk is obtained, 
in the limit $s$ tending to zero, from the operators $\widehat{H}(\psi_{s})$ defined 
through the Hamiltonians of the discrete-time quantum walks. 
However, it would be interesting to compare these operators in the weak-limits of 
the probability distributions. 
To define these probability distributions, we generalize a little bit our setting up. 
For any $k \in C(S^{1})$, we denote $M(k)$ the 
multiplication operator on $L^{2}(S^{1})$ by $k$. 
We define the operator $\widehat{H}(k)$ on $\ell^{2}(\mb{Z})$ by 
\begin{equation}\label{mult2}
\widehat{H}(k)=\fcal M(k) \fcal^{*},  
\end{equation}
which has the expression $\eqref{phis}$ with $k$ replacing $\varphi_{s}$. 
When $k$ is real-valued, the operator $\widehat{H}(k)$ is self-adjoint, and hence 
we have the $1$-parameter group of unitary operators $e^{it\widehat{H}(k)}$. 
For any $t \in \mb{R}$, we define the function $p_{t}(k;x)$ in $x \in \mb{Z}$ by 
\begin{equation}\label{PD1}
p_{t}(k;x)=|\ispa{e^{it\widehat{H}(k)}\delta_{0},\delta_{x}}|^{2} \quad (x \in \mb{Z}), 
\end{equation}
where $\ispa{,}$ denotes the standard $\ell^{2}$-inner product on $\ell^{2}(\mb{Z})$ and 
$\delta_{x} \in \ell^{2}(\mb{Z})$ is defined as 
\[
\delta_{x}(y)=
\begin{cases}
1 & \mbox{(when $y =x$)}, \\
0 & \mbox{(otherwise)}. 
\end{cases}
\]
Note that $\{p_{t}(k;x)\}_{x \in \mb{Z}}$ is a probability distribution for each $t \in \mb{R}$ and each continuous real-valued 
function $k$ on $S^{1}$. 
Our next theorem is on the weak limit distributions of the probability measures
\begin{equation}\label{contM}
d\mu_{t}(k):=\sum_{x \in \mb{Z}} p_{t}(k;x)\delta_{x/t}
\end{equation}
on $\mb{R}$ for $k=\psi_{0}$, $k=\varphi_{s}$ and $k=\psi_{s}$ ($0<s<1$), where $\delta_{\xi}$ for $\xi \in \mb{R}$ denotes 
the Dirac measure at $\xi$. 
\begin{thm}\label{wld1}
For any interval $I$ in $\mb{R}$, let $\chi_{I}(x)$ denote the indicator function of $I$. 
We set $\rho(s):=\arcsin(s)$. Then we have the following. 
\begin{enumerate}
\item $\dsp \wlim_{t \to \infty}d\mu_{t}(\psi_{0}) = \frac{1}{\pi \sqrt{1-x^{2}}} \chi_{(-1,1)}(x)\,dx$. 

\vspace{2pt}

\item $\dsp \wlim_{t \to \infty}d\mu_{t}(\varphi_{s})=\frac{r}{\pi (1-x^{2}) \sqrt{s^{2}-x^{2}}} \chi_{(-s,s)}(x)\,dx$ $(0<s<1)$. 

\vspace{2pt}

\item $\dsp \wlim_{t \to \infty}d\mu_{t}(\psi_{s})=\frac{\rho(s) r}{\pi (1-\rho(s)^{2}x^{2}) \sqrt{s^{2}-\rho(s)^{2}x^{2}}} \chi_{(-s/\rho(s),s/\rho(s))}(x)\,dx$ $(0<s<1)$.  

\end{enumerate}
\end{thm}

Note that (1) in Theorem $\ref{wld1}$ is well-known (\cite{Ko2}), and the weak limits in (2) and (3) are essentially (scaled) Konno's distribution (\cite{Ko1}), 
which appears as weak limit distributions of the transition probabilities of the one-dimensional discrete-time quantum walks. 
By Taylor's formula for $\arcsin(s)$ as $s \to 0$ one finds that the distribution in the right-hand side of (1) 
is obtained from that in (3) under the weak limit $s \to 0$. More generally, we have the following. 
\begin{thm}\label{WCont1}
For $(s,u) \in [0,1) \times [0,\infty)$, we define the probability measures $d\mu_{(s,u)}$ on $\mb{R}$ by 
\[
d\mu_{(s,u)}=
\begin{cases}
{\dsp d\mu_{1/u}(\psi_{s})} & (0 \leq s <1,\, 0<u),\\
{\dsp \frac{\rho(s)\sqrt{1-s^{2}}}{\pi (1-\rho(s)^{2}x^{2})\sqrt{s^{2}-\rho(s)^{2}x^{2}}} \chi_{(-s/\rho(s),s/\rho(s))}(x)\,dx }& (0<s<1,\,u=0),\\
{\dsp \frac{1}{\pi \sqrt{1-x^{2}}} \chi_{(-1,1)}(x)\,dx } & (s=u=0). 
\end{cases}
\]
Then, the map $(s,u) \mapsto d\mu_{(s,u)}$ defines a weakly continuous map from $[0,1) \times [0,\infty)$ to 
the space of all probability measures on $\mb{R}$ with the weak topology. 
In particular, we have 
\[
\wlim_{(s,1/t) \to (0,0)}d\mu_{t}(\psi_{s})=\frac{1}{\pi \sqrt{1-x^{2}}} \chi_{(-1,1)}(x)\,dx. 
\]
\end{thm}

Thus, the one-dimensional continuous-time quantum walk 
is obtained, under the limit $s=|a| \to 0$, from the scaled operator $\widehat{H}(\psi_{s})$ defined through 
the Hamiltonian $D(T)$ of the one-dimensional discrete-time quantum walk $U(T)$. 

\vspace{10pt}

\noindent{\bf Acknowledgment.} The fact that the algebraic structure introduced in \cite{T} 
defines a unitary representation of the infinite dihedral group, which is discussed in Section $\ref{ALG}$ in detail, 
is suggested to the author by Professor Nobuaki Obata. 
The author would like to express his special thanks to Professor Obata. 
\vspace{10pt}

\section{An algebraic structure}\label{ALG}

In our previous paper \cite{T}, we have introduced an algebraic structure behind the one-dimensional discrete-time quantum walks. 
Because we need to use this to deduce the Hamiltonian $D(T)$ in the next section, let us begin by recalling this structure. 
Suppose that we are given two unitary operators $V$, $W$ on a Hilbert space $\hcal_{0}$ satisfying the following relations. 

\vspace{5pt}

\noindent${\rm (QW1)}$ $W^{2}=-I$. 

\noindent${\rm (QW2)}$ $VW=WV^{-1}$. 

\vspace{5pt}

We remark that, in \cite{T}, we have introduced another unitary operator, $\sigma$, satisfying 
some relations with $V$ and $W$. However we do not introduce the operator $\sigma$ here because 
we will not use it. 
Now, if the positive numbers $s,r$ satisfy $s^{2}+r^{2}=1$, then 
it turns out that the linear combination 
\begin{equation}\label{unitary1}
U=sV+rW
\end{equation}
is a unitary operator on $\hcal_{0}$. We have the following examples. 

\begin{ex}\label{ex1}
{\rm 
Let $\hcal_{0}=\mb{C}^{2}$. For any $\alpha,\beta \in S^{1}$, we set 
\[
V(\alpha)=
\begin{pmatrix}
\alpha & 0 \\
0 & \ol{\alpha}
\end{pmatrix},\quad 
W(\beta)=
\begin{pmatrix}
0 & \beta \\
-\ol{\beta} & 0
\end{pmatrix}. 
\]
Then, the unitary matrices $V=V(\alpha)$, $W=W(\beta)$ satisfy the relations ${\rm (QW1)}$, ${\rm (QW2)}$. 
}
\end{ex}

\begin{ex}\label{ex2}
{\rm Let $\hcal_{0}=\ell^{2}(\mb{Z},\mb{C}^{2})$. We take $\alpha,\beta \in S^{1}$. 
Let $V$, $W$ be the quantum walks, 
\[
V=U(V(\alpha)),\quad W=U(W(\beta))
\]
defined in $\eqref{1DQWdef}$ associated with the unitary matrices $V(\alpha)$, $W(\beta)$ given in Example $\ref{ex1}$. 
Then $V$ and $W$ satisfy the relations ${\rm (QW1)}$, ${\rm (QW2)}$. 
}
\end{ex}

We note that for $T \in {\rm SU}(2)$ given by $\eqref{su2}$, the quantum walk $U(T)$ is written 
in the form $\eqref{unitary1}$ as 
\[
U(T)=sU(V(\alpha))+rU(W(\beta)), 
\]
where the parameters $s$, $r$, $\alpha$, $\beta$ are defined in $\eqref{not1}$. 
Furthermore, if we define a unitary matrix $T(z)$ with $z \in S^{1}$ by 
\begin{equation}\label{FQW1}
T(z)=sV(\alpha z) +r W(\beta \ol{z}), 
\end{equation}
then, under the Fourier transform $\fcal:L^{2}(S^{1},\mb{C}^{2}) \to \ell^{2}(\mb{Z},\mb{C}^{2})$, defined by the same formula as in $\eqref{SFourier}$, 
$U(T)$ and $T(z)$ are related each other by the formula
\begin{equation}\label{F1DQW}
\fcal^{*} U(T) \fcal =\tcal, 
\end{equation}
where $\tcal$ is the multiplication operator defined by
\[
(\tcal k)(\theta)=T(e^{i\theta}) k(\theta)\quad  (k \in L^{2}(S^{1},\mb{C}^{2})).
\]
The following is proved in \cite{T}. 
\begin{prop}\label{cheby}
Suppose that the unitary operators $V$, $W$ satisfy the relations ${\rm (QW1)}$ and ${\rm (QW2)}$. 
Let $x=\frac{s}{2}(V^{*}+V)$, $y=\frac{s}{2i}(V-V^{*})$, $w=rW$. 
Then, the $n$-th power $U^{n}$ of the unitary operator $U$ defined in $\eqref{unitary1}$ is represented as 
\[
U^{n}=T_{n}(x)+(iy+w)U_{n-1}(x), 
\]
where $T_{n}(x)$ and $U_{n-1}(x)$ are, respectively, 
the Chebyshev polynomials of the first kind of degree $n$ and the second kind of degree $n-1$. 
\end{prop}

Before explaining how to deduce the formula $\eqref{HamDef}$ of the operator $D(T)$, let us mention 
the meaning of the algebraic structure (QW1), (QW2). 
The relations (QW1), (QW2) define a unitary representation of the infinite dihedral group (see, for example, \cite{MKS}) as follows. 
The infinite dihedral group $\Gamma$ is a discrete group defined by 
\[
\Gamma=\mb{Z} \rtimes \mb{Z}_{2} \cong \mb{Z}_{2} \ast \mb{Z}_{2}, 
\]
where $\mb{Z}_{2}=\{\pm 1\}$ and it acts on $\mb{Z}$ in an obvious manner. 
As a set, $\Gamma$ is the product $\mb{Z} \times \mb{Z}_{2}$, and its group structure is given by 
\[
(x,\mu) (y,\nu)=(x+\mu y,\mu \nu) \quad ((x,\mu),(y,\nu) \in \mb{Z} \times \mb{Z}_{2}).  
\]
Hence the unit is $e:=(0,1)$ and the inverse element of $(x,\mu) \in \mb{Z} \times \mb{Z}_{2}$ is given by $(-\mu x,\mu)$. 
We set $a=(1,1)$ and $b=(0,-1)$. 
Then, $a$ and $b$ generate $\Gamma$ with the relation 
\[
ab=ba^{-1},\quad b^{2}=e. 
\]
Let $V$ and $W$ be unitary operators on a Hilbert space $\hcal_{0}$ satisfying the relations ${\rm (QW1)}$ and ${\rm (QW2)}$. 
Then, we can define a unitary representation $\rho:\Gamma \to U(\hcal_{0})$ of $\Gamma$, where $U(\hcal_{0})$ denotes 
the group consisting of all unitary operators on $\hcal_{0}$, by setting
\[
\rho(a)=V,\quad \rho(b)=-iW. 
\]
In particular, by Example $\ref{ex2}$, the quantum walks $V=U(V(\alpha))$, $W=U(W(\beta))$ define 
a unitary representation $(\ell^{2}(\mb{Z},\mb{C}^{2}),\rho_{{\scriptscriptstyle {\rm QW}}})$ of $\Gamma$. 

\begin{thm}\label{regular}
The unitary representation $(\ell^{2}(\mb{Z},\mb{C}^{2}),\rho_{{\scriptscriptstyle {\rm QW}}})$ of 
$\Gamma$ so defined is unitarily equivalent to the regular representation $(\ell^{2}(\Gamma), R)$. 
\end{thm}
\begin{proof}
We first recall that the (right) regular representation $(\ell^{2}(\Gamma),R)$, where 
$\ell^{2}(\Gamma)$ is the Hilbert space consisting of 
all square summable functions on $\Gamma$ with the standard inner product, 
is defined by the formula 
\[
R:\Gamma \to U(\ell^{2}(\Gamma)),\quad 
(R(g)f)(h)=f(hg), \quad f \in \ell^{2}(\Gamma), \ g,h \in \Gamma. 
\]
Let $\delta_{(x,\mu)}$ be an element in $\ell^{2}(\Gamma)$ defined by 
\[
\delta_{(x,\mu)}(y,\nu)=
\begin{cases}
1 & \mbox{(when $(y,\nu)=(x,\mu)$)}, \\
0 & \mbox{(otherwise).}
\end{cases}
\]
Note that $\{\delta_{(x,\mu)}\,;\,(x,\mu) \in \Gamma\}$ is an orthonormal basis of $\ell^{2}(\Gamma)$. 
Then a direct computation shows that the unitary operator $u:\ell^{2}(\Gamma) \to \ell^{2}(\mb{Z},\mb{C}^{2})$ defined by the formula
\[
u\delta_{(x,1)}=\alpha^{-x}\delta_{-x} \otimes {\bf e}_{1},\quad 
u\delta_{(x,-1)}=i\alpha^{-x}\beta^{-1} (\delta_{1-x} \otimes {\bf e}_{2}) \quad (x \in \mb{Z}) 
\]
intertwines two representations $(\ell^{2}(\mb{Z},\mb{C}^{2}), \rho_{{\scriptscriptstyle {\rm QW}}})$ and $(\ell^{2}(\Gamma), R)$. 
\end{proof}

\section{The Hamiltonian $D(T)$}\label{HAM}

Let us explain how to deduce the formulas $\eqref{intco1}$, $\eqref{kernel1}$ and $\eqref{HamDef}$ for the Hamiltonian $D(T)$. 
We apply Proposition $\ref{cheby}$ for the matrices $V=V(\alpha z)$, $W=W(\beta z^{-1})$ 
with $\alpha$ and $\beta$ given in $\eqref{not1}$ to get 
the following formula. 
\[
T(z)^{n}=
\begin{pmatrix}
p_{n}(\alpha z) & q_{n}(\alpha z) \\[5pt]
-\ol{q_{n}(\alpha z)} & \ol{p_{n}(\alpha z)}
\end{pmatrix} \quad 
(z \in S^{1}), 
\]
where $p_{n}(z)$ and $q_{n}(z)$ are the Laurent polynomials given by 
\[
p_{n}(z) = T_{n}(\tf{s}{2}(z+z^{-1})) +\frac{s}{2}(z-z^{-1})U_{n-1}(\tf{s}{2}(z+z^{-1})), \quad 
q_{n}(z) = r \alpha \beta z^{-1} U_{n-1}(\tf{s}{2}(z+z^{-1})). 
\]
Setting $\alpha=e^{i\mu}$, $z=e^{i\theta}$, and using the function $\varphi_{s}(\theta)$ defined in $\eqref{eigenF}$, 
we see that 
\begin{equation}\label{pqLpoly}
p_{n}(\alpha z) = \cos (n\varphi_{s}(\theta +\mu)) +is \sin (\theta +\mu) \frac{\sin (n \varphi_{s}(\theta +\mu))}{\sin \varphi(\theta +\mu)},\quad 
q_{n}(\alpha z) = r \beta e^{-i\theta} \frac{\sin (n \varphi_{s}(\theta +\mu))}{\sin \varphi(\theta +\mu)}. 
\end{equation}
According to the formula $\eqref{pqLpoly}$, the functions $p_{n}(\alpha z)$, $q_{n}(\alpha z)$ are differentiable with respect to $n$ at $n=0$. 
Then, differentiating each component of $T(e^{i\theta})^{n}$ in $n$ at $n=0$ and dividing them by $i$, we obtain 
the following matrix. 
\[
L(\theta)=
\frac{\varphi_{s}(\theta +\mu)}{\sin \varphi_{s}(\theta +\mu)} 
\begin{pmatrix}
s \sin (\theta +\mu) & -ir \beta e^{-i\theta} \\
ir \ol{\beta} e^{i\theta} & -s \sin (\theta +\mu)
\end{pmatrix}.
\]
Let $\lcal$ be the bounded self-adjoint operator on $L^{2}(S^{1},\mb{C}^{2})$ defined by 
\[
(\lcal k)(\theta)=L(\theta) k(\theta)\quad  (k \in L^{2}(S^{1},\mb{C}^{2})).
\]
Then, the operator $D(T)$ defined in $\eqref{HamDef}$ is related to the operator $\lcal$ by the formula 
\begin{equation}\label{HamDef2}
D(T)=\fcal \lcal \fcal^{*}. 
\end{equation}
These explanation is just a way to deduce the expression $\eqref{HamDef}$ of 
the operator $D(T)$. The proof of Theorem $\ref{HamTH}$ goes as follows. 

\vspace{10pt}

\noindent{\it Proof of Theorem $\ref{HamTH}$.} \hspace{1pt} 
From $\eqref{HamDef2}$, we see that $e^{iD(T)}=\fcal e^{i\lcal} \fcal^{*}$. Thus, according to $\eqref{F1DQW}$, 
it is enough to show that $e^{i\lcal}=\tcal$. A direct computation shows that the eigenvalues of $L(\theta)$ is $\pm \varphi_{s}(\theta+\mu)$. 
Then, the matrix $L(\theta)$ is diagonalized as 
\[
B(\theta)^{-1} L(\theta) B(\theta)=
\begin{pmatrix}
\varphi_{s}(\theta +\mu) & 0 \\
0 & -\varphi_{s}(\theta +\mu)
\end{pmatrix}
\]
with the matrix $B(\theta)$ given by 
\[
B(\theta)=
\begin{pmatrix}
ir\beta e^{-i\theta} & ir \beta e^{-i\theta} \\
s \sin(\theta +\mu) -\sin \varphi_{s}(\theta +\mu) & s \sin (\theta +\mu) +\sin \varphi(\theta +\mu)
\end{pmatrix}. 
\]
Then, one can check directly that
\[
e^{iL(\theta)}=B(\theta)
\begin{pmatrix}
e^{i\varphi_{s}(\theta +\mu)} & 0 \\
0 & e^{-i\varphi_{s}(\theta+\mu)}
\end{pmatrix}B(\theta)^{-1}
=T(e^{i\theta}),  
\]
where $T(e^{i\theta})$ is the matrix defined in $\eqref{FQW1}$. 
This shows $e^{i\lcal}=\tcal$ and hence Theorem $\ref{HamTH}$. \hfill$\square$

\vspace{10pt}

We remark that, as in the above proof, the matrix-valued function $L(\theta)$ has eigenvalues $\pm \varphi_{s}(\theta +\mu)$. 
Since the spectrum of the operator $D(T)$ coincides with that of $\lcal$ and the latter is given by the 
union of images of two functions $\theta \mapsto \pm \varphi_{s}(\theta+\mu)$, we have $\eqref{SpecDT}$. 
This fact also shows that $L(\theta)^{2}=\varphi_{s}^{2}(\theta+\mu) I_{2}$ and hence $\eqref{DTsq}$.

\section{Limit theorems}\label{PTH2}

Let $k$ be a smooth real-valued function on $S^{1}$ and consider, for each $t \in \mb{R}$, the measure $d\mu_{t}(k)$ defined in $\eqref{contM}$. 
To prove Theorem $\ref{wld1}$, we use the characteristic function $E_{t}(k;\xi)$ of $d\mu_{t}(k)$, which 
is explicitly given by the formula 
\begin{equation}\label{ch1}
E_{t}(k;\xi)=\sum_{x \in \mb{Z}} p_{t}(k;x)e^{ix \xi/t} \quad (\xi \in \mb{R}).  
\end{equation}
Let us rewrite the function $E_{t}(k;\xi)$ in a useful form. 
Recall that the convolution of two continuous functions $f,g$ on $S^{1}$ is defined by 
\[
(f \ast g)(\theta)=\frac{1}{2\pi} \int_{0}^{2\pi} f(r)g(\theta-r)\,dr. 
\]
If $f$ is a smooth function on $S^{1}$, then the Fourier series 
\[
f(\theta) =\sum_{x \in \mb{Z}} \fcal(f)(x)e^{ix \theta}
\]
of $f$ converges to $f$ itself uniformly in $\theta$, and hence, for $f,g \in C^{\infty}(S^{1})$, 
\[
(f \ast g)(\theta)=\sum_{y \in \mb{Z}} \fcal(f)(y) \fcal(g)(y)e^{iy \theta}, 
\]
which converges uniformly in $\theta$. We set $\iota(f)(\theta):=\ol{f(-\theta)}$ so that $\fcal(\iota(f))(x)=\ol{\fcal(f)(x)}$ and 
\begin{equation}\label{conv1}
(f \ast \iota(f) )(\xi)  = \sum_{x \in \mb{Z}} |\fcal(f)(x)|^{2} e^{ix \xi}. 
\end{equation}
Now we set 
\[
f(\theta)= \sum_{x \in \mb{Z}} \ispa{e^{it\widehat{H}(k)}\delta_{0},\delta_{x}}e^{ix \theta},  
\]
which is well-defined in $L^{2}(S^{1})$. 
Since $e^{it \widehat{H}(k)}=\fcal e^{itM(k)} \fcal^{*}$ and $\fcal^{*}\delta_{x}=e_{x}$, where $e_{x} \in L^{2}(S^{1})$ is 
defined by $e_{x}(\theta)=e^{ix \theta}$, we have $\ispa{e^{it \widehat{H}(k)}\delta_{0},\delta_{x}}=\ispa{e^{it M(k)}e_{0},e_{x}}$ 
and hence $f(\theta)=e^{it M(k)}e_{0}=e^{it k(\theta)}$. Thus $f$ is a smooth function on $S^{1}$, and by $\eqref{PD1}$, $\eqref{ch1}$, $\eqref{conv1}$, we see 
\begin{equation}\label{conv2}
E_{t}(k;t\xi)=(f \ast \iota(f))(\xi)=\frac{1}{2\pi}\int_{0}^{2\pi}e^{it[k(\theta) -k(\theta -\xi)]}\,d\theta.
\end{equation}

\begin{prop}\label{asych1}
Let $k$ be a smooth real-valued function on $S^{1}$. 
Then, for each $\xi \in \mb{R}$, we have
\[
\lim_{t \to \infty}E_{t}(k;\xi)=\frac{1}{2\pi} \int_{0}^{2\pi} e^{i \xi k'(\theta)}\,d\theta. 
\]
\end{prop}
\begin{proof}
By $\eqref{conv2}$ and Taylor's formula, we have 
\[
E_{t}(k;\xi)  =\frac{1}{2\pi} \int_{0}^{2\pi} e^{it[k(\theta)-k(\theta-\xi/t)]}\,d\theta 
 =\frac{1}{2\pi} \int_{0}^{2\pi} e^{it[\xi k'(\theta)/t +O(1/t^{2})]}\,d\theta \to \frac{1}{2\pi} \int_{0}^{2\pi} e^{i\xi k'(\theta)}\,d\theta
\]
as $t \to \infty$, which completes the proof.
\end{proof}

\noindent{\it Proof of Theorem $\ref{wld1}$.} \hspace{1pt} 
The items (1), (2) and (3) in Theorem $\ref{wld1}$ are easily proved by using 
Proposition $\ref{asych1}$ with $k=\psi_{0}$ for (1), $k=\varphi_{s}$ for (2) and $k=\psi_{s}$ for (3), 
respectively. \hfill$\square$

\vspace{10pt}

\noindent{\it Proof of Theorem $\ref{WCont1}$.} \hspace{1pt} 
Let $(s_{o},u_{o}) \in [0,1) \times [0,\infty)$. First assume that $u_{o}>0$. 
We set $\arcsin(s)=s(1+s^{2}a(s))$ with a smooth function $a(s)$ bounded on $0 \leq s \leq 1/2$. 
We have 
\begin{equation}\label{psiT}
\psi_{s}(\theta)=\cos \theta \frac{1+s^{2}a(s\cos \theta)}{1+s^{2}a(s)}, 
\end{equation}
which shows that the function $\psi_{s}(\theta)$ is continuous for $(s,\theta)$ with $0 \leq s$ and $\theta \in \mb{R}$. 
Thus, the characteristic function 
\[
E_{t}(\psi_{s};\xi)=
\frac{1}{2\pi}\int_{0}^{2\pi} e^{it[\psi_{s}(\theta)-\psi_{s}(\theta -\xi/t)]}\,d\theta
\]
is continuous in $s,t,\xi$ even when $s=0$. 
Hence the weak continuity of $d\mu_{(s,u)}$ at $(s_{o},u_{o})$ is obvious when $u_{o}>0$. 
Let us prove that $d\mu_{(s,u)}$ is weakly continuous at $(s_{o},0)$. 
Suppose first that $s_{o}>0$. Obviously we have $\dsp \wlim_{s \to s_{o}}d\mu_{s,0}=d\mu_{s_{o},0}$. 
Let $s,u>0$ and set $t=1/u$. Then, we have 
\begin{equation}\label{aux11}
\left|
E_{t}(\psi_{s};\xi)-\frac{1}{2\pi}\int_{0}^{2\pi} e^{i\xi \psi_{s_{o}}'(\theta)}\,d\theta 
\right|
\leq \frac{1}{2\pi}\int_{0}^{2\pi}
\left|
e^{i\xi(\psi_{s}'(\theta) -\psi_{s_{o}}'(\theta)) -i(\xi^{2}/t) \int_{0}^{1}(1-r)\psi_{s}''(\theta-r\xi/t)\,dr} -1
\right|
\,d\theta. 
\end{equation}
By $\eqref{psiT}$, we have 
\[
\psi_{s}'(\theta)=-\sin \theta (1+O(s^{2})),\quad \psi_{s}''(\theta)=-\cos \theta (1+O(s^{2})). 
\]
Hence the right-hand side of $\eqref{aux11}$ tends to zero as $(s,1/t) \to (s_{o},0)$. 
Since the characteristic function of the measure $d\mu_{(s_{o},0)}$ is given by 
\[
\xi \mapsto \frac{1}{2\pi}\int_{0}^{2\pi} e^{i\xi \psi_{s_{o}}'(\theta)}d\theta, 
\]
we see that $d\mu_{(s,u)}$ is weakly continuous at $(s_{o},0)$. 
Next, let us prove the continuity at $(0,0)$. Obviously we have $\dsp \wlim_{s \to 0}d\mu_{(s,0)}=d\mu_{(0,0)}$. 
Now let $s,u>0$ and we set $t=1/u$. 
Then, again by Taylor's formula, we obtain 
\[
\left|
e^{it[\psi_{s}(\theta)-\psi_{s}(\theta-\xi/t)]}-e^{-i\xi \sin \theta}
\right|
\leq C(s^{2}+1/t), 
\]
where the positive constant $C$ can be chosen uniformly in $\theta \in \mb{R}$ and locally uniformly in $\xi \in \mb{R}$. 
From this, we have 
\[
E_{t}(\psi_{s};\xi)=
\frac{1}{2\pi} \int_{0}^{2\pi} e^{-i\xi \sin \theta}\,d\theta+O(s^{2}+1/t)
=\frac{1}{\pi}\int_{-1}^{1}\frac{\cos (\xi x)}{\sqrt{1-x^{2}}}\,dx +O(s^{2}+1/t), 
\]
which shows $\dsp \wlim_{(s,1/t) \to (0,0)}d\mu_{(s,1/t)}=d\mu_{(0,0)}$. This completes the proof. 
\hfill$\square$

\end{document}